\documentclass[12pt]{article}     

\usepackage{amsfonts}
\usepackage{amsmath}
\usepackage{graphicx}
\usepackage{amsfonts} \usepackage{amsmath, amsthm}
\usepackage{epsfig} \usepackage{verbatim}
\usepackage{amssymb, amsthm, amsmath, bm, tikz, enumerate, color}
\usepackage{ bm, tikz, enumerate}
\usepackage{amsfonts}
\usepackage{amsmath}
\usepackage{graphicx}
\usepackage{amsfonts} \usepackage{amsmath, amsthm}
\usepackage{epsfig} \usepackage{verbatim}
\usepackage{amssymb, amsthm, amsmath, bm, tikz, enumerate, color}
\usepackage{cite}
\usepackage{todonotes}
\usepackage{xcolor}
\usetikzlibrary{spy}

\newtheorem{theorem}{Theorem}

\newtheorem{lemma}{Lemma}

\newtheorem{proposition}{Proposition}
\newtheorem{remark}[theorem]{Remark}

\newcommand{\dist}{\mathop{\mathrm{dist}}\nolimits}

\newcommand{\length}{{\mathrm{length}}}
\newcommand{\Prob}{{\mathbb{P}}}

\newcommand{\eps}{{\varepsilon}}
\newcommand{\bT}{{\mathbf{T}}}

\newcommand{\boundellipse}[3]
{(#1) ellipse (#2 and #3)
}

\title{The first encounter of two billiard particles of small radius}
\author{Dmitry Dolgopyat, P\'{e}ter N\'{a}ndori}
\begin{document}

\maketitle

\begin{abstract}
We prove that the time of the first collision between two particles in a Sinai billiard table converges weakly to 
an exponential distribution when time is rescaled by the inverse of the radius of the particles.
This results provides a first step 
in studying the energy evolution of hard ball systems in the rare interaction limit.
\end{abstract}

\section{Result}

Understanding derivation of macroscopic laws from deterministic microscopic dynamics is an outstanding 
problem in mathematical physics. So far this has been achieved in a very limited number of cases. One 
prominent example, where this program has been implemented is the ideal gas of particles moving in dispersive
domain (see \cite{BLRB00, Sz00} and references therein). Unfortunately, the results for the ideal gas do not 
conform to the predictions of statistical mechanics. The reason is that in the ideal gas there is no mechanisms
of coming to equilibrium, due to the lack of interaction between the particles. One way to rectify this situation is to
study a rare interaction limit of a many particle system. One step in this direction is to understand how often
the noninteracting particles are coming close to each other. This is the question studied in this paper.

Let
$\mathcal D = \mathbb T^2 \setminus \cup_{j=1}^J B_j$, where $B_j$ are disjoint
strictly convex sets with $\mathcal C^3$ smooth boundaries.
The phase space of one billiard particle is 
$\Omega = \mathcal D \times \mathcal S^1$. 
The billiard
flow consists of free flight among the scatterers and specular reflection off their
boundaries and is denoted by $\Phi^t: \Omega  \rightarrow \Omega $,
 $t \in \mathbb R$. The phase space of the billiard ball map is  
 $
 \mathcal M = \{ (q, v) \in \Omega: q \in \partial \mathcal D, \langle n,v  \rangle \geq 0\},
 $
 where $v$ is the normal vector of $\partial \mathcal D$ at $q$ pointing into $\mathcal D$. 
The billiard ball map $\mathcal F: \mathcal M \rightarrow \mathcal M$ takes the particle from one collision
to the next one. The flow time between two collisions is bounded from below by $\tau_{\min}>0$
and we also assume that it is bounded from above by $\tau_{\max}< \infty$ (this is called the finite horizon
condition).
The invariant measure of the billiard flow is $\mu = c_{\mu} dq dv$ and that of the billiard map is
$\nu = c_{\nu} \cos(\varphi, n) d\varphi dr$, where $c_{\mu} = \frac{1}{2 \pi |\mathcal D|}$ and
$c_{\nu} = \frac{1}{2 |\partial \mathcal D|}$. We will denote by $\Pi_q$ the projection from $\Omega$ to
$\mathcal D$ and similarly by $\Pi_{\varphi}$ the projection from $\Omega$ to $\mathcal S^1$.

Let us consider two billiard particles on the same domain,
i.e. $(q_i, \varphi_i ) \in \Omega $ for $i=1,2$. 
Without loss of generality, we assume that the first particle travels with speed one and the second one
travels with speed $\lambda \in [0,1]$ (note that by rescaling time we can fix the speed of the faster one).
Our main objective is to estimate the time we need to wait until
the two particles get $\varepsilon$ close. 
Thus we introduce the notation
$$ 
\mathcal A_{\lambda, t,\varepsilon} = 
\{ (q_1,\varphi_1), (q_2, \varphi_2): \exists s \in [0, t]: 
\| \Pi_q \Phi^s (q_1, \varphi_1) - \Pi_q 
\Phi^{\lambda s} (q_2, \varphi_2) \|  \leq \varepsilon
\}.
$$
Next, we define the rate function $\rho$ by
\begin{equation}
\label{eq:rho1}
 \rho(\lambda) =  \frac{1}{2 \pi |\mathcal D|}\int_{0}^{2 \pi} \sqrt{1 - 2 \lambda \cos \varphi + 
\lambda^2} d\varphi.
\end{equation}
Note that $\rho$ is bounded away from zero.
Our main theorem is the following

\begin{theorem}
\label{thm1}
For $Leb_1$-a.e. $\lambda \in [0,1]$ and for every fixed $T$,
$$ \lim_{\varepsilon \rightarrow 0} (\mu \times \mu) (\mathcal A_{\lambda,T/\varepsilon, \varepsilon}) = 1- e^{-\rho(\lambda) T}.
$$
\end{theorem}

\begin{remark}
Theorem \ref{thm1} was conjectured in an unpublished paper by Thomas Gilbert \cite{G15}. 
In particular, he computed 
the function $\rho$ based on the ergodicity and explicit formulas 
for the mean free path (a computation along the lines of \cite{B15}).
\end{remark}


\section{Preliminaries}
\label{sec:pre}

We introduce the basic definitions and lemmas that we will need to prove Theorem \ref{thm1}.
Fix some constant ${\mathfrak s} < \tau_{\min}$ (e.g. ${\mathfrak s} = \tau_{\min}/2$). 
In this section, we study $\Phi^{{\mathfrak s}}$, the time ${\mathfrak s}$-map of the billiard flow (that is, we only
study one particle).
The forthcoming definitions and statements concerning the time ${\mathfrak s}$-map are very similar 
to the corresponding definitions and statements for the billiard map $\mathcal F$.
As the proofs are also very similar, we do not give detailed proofs here, instead we highlight the differences
and strongly encourage the reader to consult the detailed exposition of \cite{CM06}. 

The phase space of $\Phi^{{\mathfrak s}}$ is $\Omega$.
Let us consider the Jacobi coordinates in the tangent space $\mathcal T_X \Omega$ at the point
$X = (x_1, x_2, \varphi)$:
$$ 
d \eta = \cos \varphi dx_1 + \sin \varphi dx_2, \quad 
d \xi = -\sin \varphi dx_1 + \cos \varphi dx_2, 
\quad d \omega = d \varphi.
$$
For $X = (q, \varphi) \in \Omega$ denote by $\mathcal T^{\perp}_X $ the subspace spanned by $(d \xi, d \omega)$.
Then for a.e. $X$ there exists
a $1$ dimensional unstable and a $1$ dimensional stable manifold $ W^u (X), W^s (X)$,
both of which satisfy $\mathcal T_X W^{u/s} \in  \mathcal T^{\perp}_X \Omega $.
These coincide with the stable and unstable manifolds of the flow, see Section 6.8 in \cite{CM06}.

Next, we want to introduce an extension of the class of the unstable manifolds, namely the unstable curves.
First we need to recall some notations from \cite{CM06}. The first collision in negative time is $t^-(X)$, that is
\[
t^-(X) = \max \{ t<0: \Phi^t(X) \in \mathcal M\}.
\]
Next we define the projection ${\bm P^-} (X)= \Phi^{t^-(X)} (X)$ from $\Omega$ to $\mathcal M$.
Then the linear map 
\[
D {\bm P^-}(X): \mathcal T^{\perp}_X \Omega \rightarrow \mathcal T_{t^{-}(X)} \mathcal M
\]
is a bijection. The unstable conefield of the map $\mathcal F$ is constructed in Section 4.4 of \cite{CM06}:
for any $x \in \mathcal M$, 
\[
\hat{\mathcal C}_x = \{ (dr, d\varphi) \in \mathcal T_x \mathcal M: \mathcal K \leq d\varphi / dr \leq \mathcal K
+ \cos \varphi / t^-(x)\},
\]
where $\mathcal K$ is the curvature of $\partial \mathcal D$ at the configurational component of $x$.
Now we extend this conefield to $\mathcal T \Omega$ by
\[
 \hat{\mathcal  C}_X = 
(D {\bm P^-}(X))^{-1} \hat{ \mathcal C}_{{\bm P}^-(X)} \subset
\mathcal T^{\perp}_X \Omega.
\]
Clearly, $\hat{\mathcal  C}_X$ is invariant under the flow in the usual sense.
First, we call a curve $W \subset \Omega$ {\it contact unstable curve} if at every point $X\in W$,
the tangent line $\mathcal T_X W$ belongs to the cone $\hat{\mathcal  C}_X$. Although the contact 
unstable curves would suffice for the sake of the present work, we prefer to slightly generalize 
the concept. Thus we introduce
\[
 {\mathcal  C}_X = \{ (d \eta, d \xi, d \omega) \in \mathcal T_X \Omega: 
(d \xi, d \omega) \in \hat{\mathcal  C}_X, d \eta / d \xi < C_f
\},
\]
where $C_f$ is a fixed universal constant ($f$ stands for flow). Also note that for tangent vectors
in ${\mathcal  C}_X$, $d \eta / d \omega < C'_f$ 
holds with some universal constant $C'_f$.
We
say that 
$W \subset \Omega$ is an {\it unstable curve} if at every point $X\in W$, 
$\mathcal T_X W \in \mathcal C_X$.
The image of an unstable
curve is unstable. Furthermore, unstable curves are stretched by the map $\Phi^{\mathfrak s}$ 
in the sense that
\begin{equation}
\label{eq:stretch}
 \mathcal J_W (\Phi^{\mathfrak s} )^n (X) :=
 \frac{\| D_X \left( \Phi^{\mathfrak s}\right)^n (dX) \|}{\| dX \|} \geq C \Lambda ^n
\end{equation}
with some $\Lambda= \Lambda(C_f)>1$ 
uniformly in $X$ and $dX \in \mathcal C_X$. Here, $\mathcal J$ stands for Jacobian
and $\Lambda$ is a constant that only depends on $\mathcal D$ and $C_f$
(it also depends on $\mathfrak{s}$ but we ignore this dependence here since
we have chosen $\mathfrak{s}=\tau_{min}/2$).

The expansion is bounded from below but it is unbounded from above near grazing collisions.
So as to recover distortion bounds and following the common approach introduced by Bunimovich,
Chernov and Sinai (\cite{BSCh90}),
we decompose the phase space $\Omega$ into homogeneity domains ${\bm G}, {\bm H}_k$,
with $k=0$ or $|k| \geq k_0$. Namely, 
for such $k$'s, we define
\[
 {\bm H}_{k} = \{ X \in \Omega: t^-(X) < {\mathfrak s} \text{ and } {\bm P}^- (X) \in \mathbb H_k\}
\]
where 
\[
\mathbb H_{k} =
\begin{cases}
 \{ (r, \varphi): - \pi /2 + k_0^2 < \varphi < + \pi /2 - k_0^2 \} &\text{ for $k=0$}\\
\{ (r, \varphi): \pi /2 - k^2 < \varphi < + \pi /2 - (k+1)^2 \} &\text{ for $k\geq k_0$}\\
\{ (r, \varphi): - \pi /2 + (k+1)^2 < \varphi < - \pi /2 + k^2 \} &\text{ for $k\leq -k_0$}\\
\end{cases}
\]
are the usual homogeneity strips of the map $\mathcal F$. Finally, we define 
${\bm G} = \{ X \in \Omega:  t^-(X) > {\mathfrak s} \}$. We say that an unstable curve $W$
is {\it weakly homogeneous}, if it belongs to one homogeneity domain. 
Then we have the following distortion bound (cf. Lemma 5.27 in \cite{CM06}): 
if $(\Phi^{\mathfrak s})^{- n} W$ is weakly homogeneous
for every $0 \leq n \leq N-1$, then for all $1 \leq n \leq N$ and for all $Y, Z \in W$,
\begin{equation}
\label{eq:distorsion}
e^{- C_d \frac{|W(Y,Z)|}{ |W|^{2/3}}} \leq \frac{ \mathcal J_W (\Phi^{\mathfrak s} )^{-n} (Y)  }{\mathcal J_W (\Phi^{\mathfrak s} )^{-n} (Z) } 
\leq e^{ C_d \frac{|W(Y,Z)|}{ |W|^{2/3}}}
\end{equation} 
with some constant $C_d$ depending on $\mathcal D$ and $C_f$. Here, $|W(Y,Z)|$ is the length of the segment
of $W$ lying between $Y$ and $Z$. The proof of (\ref{eq:distorsion}) is analogous to that of Lemma 5.27 in 
 \cite{CM06} (observe that  $\frac{d}{d\tilde X} \ln \mathcal J_{\tilde W} (\Phi^{\mathfrak s})^{-1} (\tilde X)$ is bounded 
 if $\tilde W \subset {\bf G}$, other cases are analogous to (5.8) in \cite{CM06}).

Next, we define {\it standard pairs} for the map $\Phi^{{\mathfrak s}}$. A standard pair $\ell = (W, \rho)$ consist of 
a weakly homogeneous unstable curve $W$ and a probability density $\rho$ supported on $W$ which satisfies 
\begin{equation}
\label{eq:stpair}
\left| \ln \frac{d \rho}{d Leb} (X) - \ln \frac{d \rho}{d Leb} (Y) \right| \leq C_r \frac{|W(X,Y)|}{ |W|^{2/3}}
\end{equation}
where $C_r$ is some fixed big constant.

Now we are ready to state a key Lemma, which is often called the growth lemma.
\begin{lemma}
\label{lem:gr}
Let $\ell=(W, \rho)$ be a standard pair and $A$ a measurable set. Then
\begin{equation}
\label{eq:Markov}
 \mathbb E_\ell (A \circ (\Phi^{\mathfrak s})^n)
= \sum_{a} c_{a,n} \mathbb E_{\ell_{an}} (A),
\end{equation}
where $c_{a,n} >0$, $\sum_{a} c_{a,n} =1$; $\ell_{an}
= (W_{an}, \rho_{an})$ are standard pairs such that
$\cup_a W_{an} = (\Phi^{{\mathfrak s}})^n W$ and $\rho_{an}$ is the
push-forward of $ \rho$ by $(\Phi^{{\mathfrak s}})^n$ up to a multiplicative
constant. Finally, there are  constants $\varkappa,
C_1$ (depending
on $\mathcal D$ and $C_f$), such that
if $n > \varkappa |\log \length (W)|$, then
\begin{equation}
\label{EqGrowth}
 \sum_{\length (\ell_{an}) < \varepsilon} c_{a,n}
< C_1 \varepsilon. 
\end{equation}
\end{lemma}
To fix terminology, we will call \eqref{eq:Markov} a Markov decomposition.

The difference between the proof of Lemma \ref{lem:gr} and analogous lemmas for $\mathcal F$ (Section 5.10 and 7.4
in \cite{CM06}) is slightly more substantial than in case of the previous statements. Namely, the proof in 
\cite{CM06} for the case of
$\mathcal F$ is based on the one-step expansion estimate
\begin{equation}
\label{eq:1stepexp}
\liminf _{\delta \rightarrow 0} \sup_{\mathcal W: |\mathcal W|< \delta } \sum_i \lambda (\mathcal W_i) <1,
\end{equation}
where $\mathcal W$ is an unstable curve for the map $\mathcal F$, 
$\mathcal W_i$ are the H-components of $\mathcal F(\mathcal W)$ and $\lambda (\mathcal W_i)$ 
is the maximal contraction factor of 
$\mathcal F^{-1}$ on $\mathcal W_i$. (These have similar definitions to ours, see \cite{CM06}). 
We note that \eqref{eq:1stepexp} cannot hold for our case. Indeed, if $W$ is such an unstable curve
that some of its points experience a near perpendicular 
collision within time $\mathfrak s$ with a scatterer of small curvature and some other 
points do not collide within time 
$\mathfrak s$, than the one-step expansion is violated (even with an adapted norm). That is why we prove
the $N$-step expansion instead (with some suitable $N$):
\begin{equation}
\label{eq:Nstepexp}
\liminf _{\delta \rightarrow 0} \sup_{W: |W|< \delta } \sum_i \lambda (W_{i,N}) <1,
\end{equation}
where $W_{i,N}$ are the H-components of $(\Phi^{\mathfrak s})^N W$ 
and $\lambda (W_i)$ is the maximal contraction factor of 
$(\Phi^{\mathfrak s})^{-N} $ on $W_{i,N}$. 
We note that similar $N$-step expansions have been used several times, e.g. in case of billiards with corner
points \cite{Ch99, DST14}.

The proof of \eqref{eq:Nstepexp}
relies on a complexity estimate which we derive next.
First, 
we write $\Omega = \bm H \cup \bm G$ (modulo a set of zero $\mu$ measure), where 
$\bm H = \{ X \in \Omega: t^-(X) < \mathfrak s\}$ (and consequently
$\overline{ \bm H} = \cup \overline{ \bm H_k}$).
Note that $\Phi^{\mathfrak s}$ is continuous on the domains $\bm H$ and $\bm G$. 
For an unstable curve $W$ we say that a sequence $(A_1, ..., A_n) \in \{ \bm H, \bm G\}^n$ is admissible
if there is some point $X \in W$ such that $(\Phi^{\mathfrak s})^i \in A_i$ for all $i = 1,...,n$.
Next, we define
\[
K_n(\delta)  =\max_{W: |W| < \delta} \# \{ \text{admissible sequences } (A_1, ..., A_n)\}.
\]
Now our complexity bound is the following. There exists some $L<\infty$ such that for all $n >0$,
\begin{equation}
\label{complexity}
\lim_{\delta \rightarrow \infty} K_n(\delta) < L n^2.
\end{equation}
In order to derive \eqref{complexity}, we first recall the complexity bound of the map $\mathcal F$ from \cite{BSCh90}.
The complexity
$\mathcal K_n(\delta)$ of the map $\mathcal F$ (that is the maximal number the singularity set 
$\{ \varphi = \pm \pi /2\}$ can cut an unstable curve of length $\leq \delta$ during $n$ iteration) satisfies
\begin{equation}
\label{complexityold}
\lim_{\delta \rightarrow \infty} \mathcal K_n(\delta) < A n.
\end{equation}
There is two reasons why $(\Phi^{\mathfrak s})^n$ can cut an unstable curve: grazing collisions and 
collisions at times which are integer multiples of $\mathfrak s$. The grazing collisions are liable for the fragmentation
of unstable curves of the map $\mathcal F$, hence can be bounded by \eqref{complexityold}.
Namely, for fixed $n$ and for small enough $\delta$, $W$ can be cut into pieces $W_{\alpha}$, $\alpha \leq Bn$
such that all $X,Y$ belonging to the same piece $W_{\alpha}$ 
collide on the same sequence of scatterers during flow time $n \mathfrak s$. 
Indeed, as the number of collisions during flow time $n \mathfrak s$ is bounded by 
$n \mathfrak s / \tau_{\min}$, we can choose $B = A \mathfrak s / \tau_{\min}$.
Now pick some $W_{\alpha}$. Observe that by definition, 
$\Pi_{q} (\Phi^{\mathfrak s})^i W_{\alpha}$,
the projection of the unstable curve $(\Phi^{\mathfrak s})^i W_{\alpha}$ to the configuration space, is convex.
Consequently there can be at most $2$ points on $ W_{\alpha}$ which collide exactly at time $(i+1) \mathfrak s$.
We conclude that each $W_{\alpha}$ is cut into at most $2n+1$ pieces which proves \eqref{complexity}
(with say $L = 3A$).

The derivation of \eqref{eq:Nstepexp} from \eqref{complexity} goes along the lines of the proof of Lemma 5.56 in
\cite{CM06}. (We need to choose $N$ such that $LN^3 < \Lambda^N$ so as to bound $\sum \lambda (W_{i,N})$
for $i$'s which never visited the nearly grazing domains $\bm H_k$ with $|k| \geq k_0$, and choose 
$k_0$ big to bound the remaining terms.) Finally, the proof of Lemma \ref{lem:gr} based on
\eqref{eq:Nstepexp} is again similar to the usual argument (see also \cite{Ch99, DST14}).

For standard pairs $\ell = (W, \rho)$, we will write $\mathbb E_{\ell}$ for 
the integral with respect to $\rho$ and $\mathbb P_{\ell} (A) = \mathbb E_{\ell} (1_{A})$.
A standard family is a weighted average of standard pairs: $\mathcal G =
(W_{\alpha}, \rho_{\alpha})$, $\alpha \in \mathfrak A$ and a measure 
$\lambda_{\mathcal G}$ on the (possibly infinite) index set $\mathfrak A$.
The $\mathcal Z$-function of $\mathcal G$ is defined by 
\[
\mathcal Z_{\mathcal G} = \sup_{\varepsilon >0} \frac{\int_{\mathfrak A} 
\mathbb P_{\ell_{\alpha}} (r_{\mathcal G} < \varepsilon)
d\lambda_{\mathcal G}}{\varepsilon}.
\]
Here $r_{\mathcal G}(X)$ is the distance of $X$ and the closest endpoint of $W_{\alpha}$, 
where $W_{\alpha} \ni X$. With these notations, Lemma
\ref{lem:gr} can be shortly reformulated by saying that $\mathcal G_n$, the image of a standard family $\mathcal G$
under $(\Phi^{\mathfrak s})^n$
is a standard family
and there are constants $\vartheta <1$, $\beta_1$ and $\beta_2$ (depending on $\mathcal D$ and $C_f$) such that
$\mathcal Z_{\mathcal G_n} \leq \beta_1 \vartheta^n \mathcal Z_{\mathcal G} + \beta_2$.

One very important example for a standard family is the decomposition of the SRB measure $\mu$ 
to local unstable manifolds. The fact that the $\mathcal Z$ function is finite is far from being obvious, but 
it follows from
early works of Sinai. (See Theorem 5.17 in \cite{CM06} for the case of $\mathcal F$; our case is analogous).

We conclude this section with a stretched exponential bound on the decay of correlations with respect to a 
standard pair. First, we recall a few definitions from \cite{Ch07}. For some function $F : \Omega \rightarrow 
\mathbb R$, $x \in \Omega$ and $r>0$, we write 
$\mathrm{osc}_r(F,x) = \sup_B F - \inf_B F $, where $B$ is the ball of radius
$r$ centered at $x$. We say that $F$ is generalized H\"older continuous with exponent $\alpha \in (0,1]$ if
\[
\| F\|_{\alpha} = \sup_r r^{-\alpha} \int_{\Omega} \mathrm{osc}_r (F,x) d \mu (x) < \infty,
\]
and write $\mathrm{var}_{\alpha} (F) = \| F \|_{\alpha} + \sup_{\Omega} F - \inf_{\Omega} F$.

\begin{theorem}
\label{thm:DEC} 
Let $\ell = (W,\rho)$ be a standard pair and 
let $F:\Omega \to \mathbb R$ be generalized H\"older continuous with parameter $\alpha$ and 
$\int_{\Omega} F d \mu=0$. Then
 $$\left| \mathbb E _{\ell} (F\circ \Phi^t ) \right| \le |W|^{-1} C \mathrm{var}_{\alpha} (F) e^{-a \sqrt{t}},$$
 with
constants $C, \alpha $ depending only on $\mathcal D$, $C_f$ and $\alpha$.
\end{theorem}

A 
sketch of proof of Theorem \ref{thm:DEC} was given by Chernov in \cite{Ch07} for contact unstable curves
(a more detailed proof will be provided in \cite{B16}).
For general unstable curves, one can apply the same smoothening as in Corollary 1.2 of 
\cite{Ch07} after chopping $W$ to pieces of length $\leq \varepsilon/C_f$.
We note that a recent paper \cite{BDL15} obtains exponential mixing for smooth observable. 
Using this results it seems likely that the bound of Theorem \ref{thm:DEC} can be improved to
exponential. However, we do not pursue this question here since the bounds of \cite{Ch07} and \cite{B16}
are sufficient for our purposes.


\section{Proof of Theorem 1}

\subsection{Idea of the proof}
We will say that a time $t$ is microscopic if $t \leq 1$, mesoscopic if $1<t \leq \delta'/\varepsilon $
and macroscopic if $\delta' / \varepsilon <t$.

For most of the proof, we fix $(q_1, \varphi_1)$ and use the results of Sinai billiards (cf. Section
\ref{sec:pre}) for the second particle.

Recall that a random variable $\bT$ has exponetial distribution with parameter $\rho$ if
\begin{equation}
\label{InfinitesimalExp}
\Prob(\bT\in [t, t+\delta']|\bT>t)=\rho \delta'(1+o(1)). 
\end{equation}
Therefore in our setting we need to show that probability that the first close encounter
happens during the time $[\frac{t}{\eps}, \frac{t+\delta'}{\eps}]$ given that there was no collision
in the past equals to $\rho(\lambda)\delta'(1+o(1)).$ More precisely, in order to ensure the near independence
of consecutive intervals it is convenient to introduce short buffer zone between them. The fact that
the first collision is unlikely to fall to a buffer zone then follows by Markov inequality.
To estimate the collision probability during an interval of size $\delta'/\eps$ we divide it into intervals
of length $\delta\ll 1.$ Using elementary geometry
we show that the measure of the trajectories having a close encounter during such an interval
equals to $\rho(\lambda)(1+o(1)).$ Summing the probabilities along all short intervals in a given
interval of size $\delta'/\eps$ we get $\rho(\lambda)\delta'.$  
Thus to obtain \eqref{InfinitesimalExp} for our system we need to show that recollisions
have smaller order if $\delta'$ is small. To prove this we first estimate a measure of trajectories having
more than one small encounter and then use mixing to accommodate conditioning on the past.
The recollisions happening during relatively separated times (that is if the collision times
are at least $\ln^{100} \eps$ apart) are ruled out by mixing, while fast recollisions are handled
by a geometric argument. Namely we show that such recollisions are easily destroyed if we tune
the speed of the second paprticle. This is the only part of the proof which does not work for {\bf all}
values of $\lambda.$ At this step we also need to exclude the encounters where the particles 
either have almost parallel velocities or are close to the scatterers.

Specifically, we introduce
\begin{eqnarray*}
\mathcal A^{\xi}_{\lambda, t,\varepsilon} &=& 
\{ (q_1,\varphi_1), (q_2, \varphi_2): \exists s \in [0, t]: \\
&&
\| \Pi_q \Phi^s (q_1, \varphi_1) - \Pi_q 
\Phi^{\lambda s} (q_2, \varphi_2) \|  \leq \varepsilon,\\
&& \dist(\Pi_q \Phi^s (q_1, \varphi_1), \partial \mathcal D) > \xi,\\
&& | \Pi_ {\varphi} \Phi^s (q_1, \varphi_1) - \Pi_{\varphi} 
\Phi^{\lambda s} (q_2, \varphi_2) \quad (\text{mod }\pi) |  > \xi
\}.
\end{eqnarray*}
and
\begin{equation}
\label{eq:Afixedq1v1v0}
\mathcal A_{\lambda, \delta, \varepsilon}^{\xi} ( q_1,\varphi_1) 
= \{ (q_2,\varphi_2) : (q_1,\varphi_1,q_2,\varphi_2) \in  A_{\lambda, \delta, \varepsilon}^{\xi}\}.
\end{equation}


The order of choice of the parameters  can be summarized as
$$ \varepsilon \ll \delta \ll  \delta' \ll \xi \ll 1$$
(for each inequality $ \ll $ we impose finitely many upper bounds along the way, it is possible to take 
the smallest one).

To simplify notation, we will say that the two particles have a "good collision" at time $s$ if the last 
three lines of the definition of $\mathcal A^{\xi}_{\lambda, t,\varepsilon}$ are true
(note that this is not a real collision, and the notion is dependent on $\lambda$, $\varepsilon$ and $\xi$).

\begin{remark}
We note that the strategy of using \eqref{InfinitesimalExp} for proving exponential distribution for hitting times to
small sets with relying on mixing to handle the mesoscopic return times and on geometry to handle short return
times is by now pretty standard in the dynamical systems literature (see e.g. \cite{CC13, D04, H14, R14} 
and reference therein), however the implementation depends very much on the system at hand.
\end{remark}

\subsection{Microscopic and mesoscopic time}

Our first lemma concerns microscopic timescales:
\begin{lemma}
\label{lem1}
For all $\lambda \in (0,1]$,
$$\lim_{\xi \rightarrow 0}
\lim_{\delta \rightarrow 0} \frac{1}{\delta} \lim_{\varepsilon \rightarrow 0} \frac{1}{\varepsilon}
(\mu \times \mu) (\mathcal A^{\xi}_{\lambda,\delta,\varepsilon})
= \rho(\lambda)$$
\end{lemma}

\begin{proof} 
Clearly, we can assume that $\dist( q_1, \partial \mathcal D )> 2 \xi$ (which has $\mu$-measure $1-O(\xi)$). 
Then, if the 
two point particles collide within time $\delta$, then necessarily 
$\dist( q_2, \partial \mathcal D )> \xi$ which means that none of the particles collide with
the boundary of the billiard table within time $\delta$ and in particular there can be no more
than one binary collision. 
Now we fix $q_1$ as above and w.l.o.g. write $\varphi_1 = 0$. Then we compute the two dimensional measure of the surface
\begin{equation}
\label{eq:Afixedq1v1}
\mathcal A_{\lambda,\delta, 0}^{\xi}(q_1, 0) = \{(q_2,\varphi_2): (q_1,0,q_2,\varphi_2 )\in \mathcal A^{\xi}_{\lambda,\delta,0}\}
\subset \mathbb R^3.
\end{equation}

The collision takes place at time $t \in [0,\delta]$. Now we have the following parametrization of this surface:
\begin{eqnarray*}
&& u: [0, \delta] \times I_{\xi} \rightarrow \mathbb R^3 \text{ with } 
I_{\xi} = ([\xi, \pi - \xi] \cup [\pi + \xi, 2 \pi - \xi]), \\
&&u(t,\varphi_2) = (t-\lambda t \cos \varphi_2, \lambda t \sin \varphi_2, \varphi_2) + (q_1, 0)
\end{eqnarray*}
Then 
\begin{eqnarray*}
 &&Leb_2(\mathcal A_{\lambda,\delta, 0}^{\xi}(q_1, 0)) \\
&=& \int_{t=0}^{\delta} \int_{\varphi_2 \in I_{\xi}} \sqrt{1 - 2 \lambda \cos \varphi_2 + 
\lambda^2 + \lambda^2 t^2 \cos^2 \varphi_2 - 2 \lambda^3 t^2 \cos \varphi_2 + \lambda^4 t^4} d \varphi_2 dt\\
&\sim& \delta 
\int_{\varphi_2 \in I_{\xi} } \sqrt{1 - 2 \lambda \cos \varphi_2 + 
\lambda^2} d\varphi_2 
\end{eqnarray*}
as $\delta \rightarrow 0$.
Now the asymptotics for $\varepsilon \rightarrow 0$
\begin{eqnarray*}
 &&\mu(q_2, \varphi_2: (q_1,0,q_2,\varphi_2 )\in \mathcal A^{\xi}_{\lambda,\delta,\varepsilon}) \\
&=& \frac{1}{2 \pi |\mathcal D|} Leb_3(q_2, \varphi_2: (q_1,0,q_2,\varphi_2 )\in \mathcal A^{\xi}_{\lambda,\delta,\varepsilon}) \\
&\sim&  
\frac{1}{2 \pi |\mathcal D|} \varepsilon Leb_2(\mathcal A_{\lambda,\delta, 0}^{\xi}(q_1, 0)) 
\end{eqnarray*}
completes the proof.
\end{proof}


The following more technical version of Lemma \ref{lem1} 
also follows from the above proof.

\begin{lemma}
 \label{lem1b}
For any $\eta >0$ there is some $\xi_0$ such that for all $\xi < \xi_0$
there is some $\delta_0 = \delta_0(\eta, \xi)$ such that for all $\delta < \delta_0$
there is some $\varepsilon_0 = \varepsilon_0(\eta, \xi, \delta)$ such that for all $\varepsilon < \varepsilon_0$
\begin{itemize}
 \item[(a)] for all $q_1, \varphi_1$, $\mu \left( \mathcal A^{\xi}_{\lambda,\delta,\varepsilon}(q_1, \varphi_1) \right)
< (\rho(\lambda) + \eta) \delta \varepsilon$.
 \item[(b)] for all $q_1$ with $\dist( q_1, \partial \mathcal D )> 2 \xi$ and for all $\varphi_1$, 
$\mu \left( \mathcal A^{\xi}_{\lambda,\delta,\varepsilon}(q_1, \varphi_1) \right)
> (\rho(\lambda) - \eta) \delta \varepsilon$.
\end{itemize}
\end{lemma}



The next lemma bounds the probability of short return and is of crucial importance.

\begin{lemma}
\label{lem3}
For $Leb_1$-a.e. $\lambda \in [0,1]$, 
\begin{equation*}
\lim_{\varepsilon \rightarrow 0} 
\frac{(\mu \times \mu) (A^{\xi}_{\lambda,\delta,\varepsilon} \cap 
(\Phi^{-\delta} \times \Phi^{-\lambda \delta}) A^{\xi}_{\lambda,\log^{100} \varepsilon ,\varepsilon})}
{\varepsilon^{1.99}} = 0
\end{equation*}
\end{lemma}

\begin{proof}
First observe that
$$
A^{\xi}_{\lambda,\delta,\varepsilon} \cap 
(\Phi^{-\delta} \times \Phi^{-\lambda \delta}) A^{\xi}_{\lambda,\log^{100} \varepsilon ,\varepsilon}
\subset
\cup_{k=0}^{\delta/\varepsilon}
(\Phi^{-k \varepsilon} \times \Phi^{-k \lambda \varepsilon}) 
(A^{\xi}_{\lambda,\varepsilon,\varepsilon} \cap 
(\Phi^{-\delta + k \varepsilon} \times \Phi^{\lambda (-\delta + k \varepsilon)}) 
A^{\xi}_{\lambda,\log^{100} \varepsilon ,\varepsilon})
$$
Using this, the invariance of $\mu$ and the fact that two "good collisions"
(in the sense defined after $\mathcal A^{\xi}_{\lambda,t,\varepsilon}$) 
are necessarily separated by $\xi$ we conclude
\begin{eqnarray}
&&(\mu \times \mu) (A^{\xi}_{\lambda,\delta,\varepsilon} \cap 
(\Phi^{-\delta} \times \Phi^{-\lambda \delta}) A^{\xi}_{\lambda,\log^{100} \varepsilon ,\varepsilon}) 
\nonumber \\
&\leq& \frac{\delta}{\varepsilon} 
( \mu \times \mu )
(A^{\xi}_{\lambda,\varepsilon,\varepsilon} \cap 
(\Phi^{-\xi} \times \Phi^{-\lambda \xi})
A^{\xi}_{\lambda,\log^{100} \varepsilon ,\varepsilon} ) \nonumber \\
& \leq & 
\frac{\delta}{\varepsilon} 
( \mu \times \mu )
(
\{ \| q_1- q_2\| < 3 \varepsilon\}
\cap 
(\Phi^{-\xi} \times \Phi^{-\lambda \xi})
A^{\xi}_{\lambda,\log^{100} \varepsilon ,\varepsilon} ).
 \label{eq:return1}
\end{eqnarray}

Next we prove that for arbitrary $q_1, \varphi_1, q_2, \varphi_2$ fixed,
\begin{equation}
\label{eq:keylambda}
  Leb_1 (\lambda: (q_1, \varphi_1, q_2, \varphi_2) \in (\Phi^{-\xi} \times \Phi^{-\lambda \xi}) 
  A^{\xi}_{\lambda, \log^{100} \varepsilon ,\varepsilon}) <
C_{\xi} \varepsilon \log^{200} \varepsilon.
\end{equation}
Since, by the definition of a good collision, the particles come to a close encounter with
transversal velocities, their paths should intersect near the time of a close encounter.
We note that as the free flight is bounded from below, 
the trajectories in the configuration space
$\Pi_q \{ \Phi^t (q_i, \varphi_i) \}_{t \in [\xi,\log^{100} \varepsilon]}$
can have at most $\log^{200} \varepsilon$ "good" intersections (where good means that their angle is at least $\xi$
and their distance from the boundary is at least $\xi /2$). 
Let us denote the time instants when the first particle arrives at these intersections by 
$t_i, i< \log^{200} \varepsilon  $. Now a simple geometry shows that 
$(q_1, \varphi_1, q_2, \varphi_2) \in (\Phi^{-\xi} \times \Phi^{-\lambda \xi})
\mathcal A^{\xi}_{\lambda,\log^{100} \varepsilon ,\varepsilon} $
implies
\begin{equation}
\label{keylambda2}
\| \Pi_q \Phi^{t_k} (q_1, \varphi_1) - \Pi_q \Phi^{\lambda t_k} (q_2, \varphi_2) \| < \frac{2 \varepsilon}{ \sin \xi} 
\quad \text{for some $k < \log^{200} \varepsilon$.}
\end{equation}
Now the set of $\lambda$'s satisfying (\ref{keylambda2}) for a fixed $k$ is an interval whose 
length is bounded by $\frac{2 \varepsilon}{t_k \sin \xi} < C_{\xi} \varepsilon$. (\ref{eq:keylambda}) follows.


Combining (\ref{eq:return1}) and (\ref{eq:keylambda}) we obtain
\begin{eqnarray*}
&&
\int_{\lambda} {(\mu \times \mu) (\mathcal A^{\xi}_{\lambda,\delta,\varepsilon} \cap 
(\Phi^{-\delta} \times \Phi^{-\lambda \delta}) \mathcal A^{\xi}_{\lambda,\log^{100} \varepsilon ,\varepsilon})}
d Leb_1(\lambda) \\
&\leq& \frac{\delta}{\varepsilon} \int_{\lambda} 
\int_{\{ \| q_1- q_2\| < 3 \varepsilon\}, \varphi_1, \varphi_2} 1_{(\Phi^{-\xi} \times \Phi^{-\lambda \xi})
\mathcal A^{\xi}_{\lambda,\log^{100} \varepsilon ,\varepsilon} } d (\mu \times \mu) dLeb_1( \lambda)\\
&\leq&C_{\xi, \delta} \varepsilon^2 \log^{200} \varepsilon.
\end{eqnarray*}

The Markov inequality gives
\begin{equation}
Leb_1 \{ \lambda: 
(\mu \times \mu) (\mathcal A^{\xi}_{\lambda,\delta,\varepsilon} \cap 
(\Phi^{-\delta} \times \Phi^{-\lambda \delta}) \mathcal A^{\xi}_{\lambda,\log^{100} \varepsilon ,\varepsilon}) 
> {\varepsilon^{1.995}}
\} < {C_{\xi, \delta} \varepsilon^{0.005} \log^{200} \varepsilon} \label{eq:keylemmaMarkov}
\end{equation}
Clearly, (\ref{eq:keylemmaMarkov}) holds with $\log^{100} \varepsilon $ replaced by 
$\log^{100} (2 \varepsilon)$ (possibly with some new constant $C_{\xi, \delta} $). Thus we have for all 
positive integer $l$,
\begin{eqnarray*}
&&Leb_1 \{ \lambda: \exists \varepsilon \in [2^{-l}, 2^{-l-1}];
(\mu \times \mu) (\mathcal A^{\xi}_{\lambda,\delta,\varepsilon} \cap 
(\Phi^{-\delta} \times \Phi^{-\lambda \delta}) \mathcal A^{\xi}_{\lambda,\log^{100} \varepsilon ,\varepsilon}) 
>{\varepsilon^{1.995}} \} \nonumber \\
&& < {C'_{\xi, \delta} 2^{-0.005 l} l^{200}} \label{eq:keylemmaMarkov2}
\end{eqnarray*}
Since this is summable in $l$, the Borel Cantelli lemma completes the proof.
\end{proof}


\subsection{Macroscopic time}
Now we turn to macroscopic time.
\begin{lemma}
\label{lemma:delta'}
For $Leb_1$-a.e. $\lambda \in [0,1]$,
$$ 
\lim_{\xi \rightarrow 0}
\lim_{\delta' \rightarrow 0} \frac{1}{\delta'}
 \lim_{\varepsilon \rightarrow 0}
(\mu \times \mu)
 \mathcal A^{\xi}_{\lambda, \frac{\delta'}{\varepsilon}, \varepsilon} = \rho(\lambda)
$$
\end{lemma}

\begin{proof}
Let us write
$$ \mathcal A^{\xi}_{\lambda, \frac{\delta'}{\varepsilon}, \varepsilon} 
= \cup_{k=1}^{\frac{\delta'}{\varepsilon \delta} -1} \mathcal C_k,$$
where $\mathcal C_k = (\Phi^{-k \delta} \times \Phi^{-\lambda k \delta}) \mathcal  A^{\xi}_{\lambda,\delta,\varepsilon}.$
Then the bound
$ (\mu \times \mu)
 \mathcal A^{\xi}_{\lambda, \frac{\delta'}{\varepsilon}, \varepsilon} 
   \leq \frac{\delta'}{\varepsilon \delta}   (\mu \times \mu)
 \mathcal C_0$
and Lemma \ref{lem1} give
$$ 
\lim_{\xi \rightarrow 0}
 \lim_{\varepsilon \rightarrow 0}
(\mu \times \mu)
 \mathcal A^{\xi}_{\lambda, \frac{\delta'}{\varepsilon}, \varepsilon} 
 \leq \delta'  \rho (\lambda),
$$
whence the upper bound follows.

To derive the lower bound, we write
$$ (\mu \times \mu)
 \mathcal A^{\xi}_{\lambda, \frac{\delta'}{\varepsilon}, \varepsilon} 
   \geq \frac{\delta'}{\varepsilon \delta}   (\mu \times \mu)
 \mathcal C_0 - \sum_{0 \leq k_1 < k_2 < \frac{\delta' }{\varepsilon \delta} -1}
 (\mu \times \mu) ( \mathcal C_{k_1} \cap \mathcal C_{k_2})$$
 An analogous argument to the upper bound will prove
 the lower bound once we establish
 \begin{equation}
 \label{eq:lem:delta'lowerbd}
 \lim_{\xi \rightarrow 0}
 \lim_{\delta' \rightarrow 0} \frac{1}{\delta'}
 \lim_{\delta \rightarrow 0}
 \lim_{\varepsilon \rightarrow 0}
 \sum_{0 \leq k_1 < k_2 < \frac{\delta' }{\varepsilon \delta}-1}
(\mu \times \mu)
( \mathcal C_{k_1} \cap \mathcal C_{k_2}) = 0.
 \end{equation}
First, using the invariance of $\mu$ we have
\begin{equation}
 \label{eq:lem:delta'lowerbd2}
\sum_{0 \leq k_1 < k_2 < \frac{\delta' }{\varepsilon \delta}-1}
(\mu \times \mu)
( \mathcal C_{k_1} \cap \mathcal C_{k_2})
\leq 
\frac{\delta' }{\varepsilon \delta}
\sum_{1 \leq k \leq \frac{\delta' }{\varepsilon \delta}}
(\mu \times \mu)
( \mathcal C_{0} \cap \mathcal C_{k}).
\end{equation}
The short returns are guaranteed to have small contribution by Lemma \ref{lem3}:
\begin{equation}
 \label{eq:lem:delta'lowerbd3}
\lim_{\varepsilon \rightarrow 0} \frac{1}{\varepsilon }
\sum_{1 \leq k \leq (\log^{100} \varepsilon)/\delta}
(\mu \times \mu)
( \mathcal C_{0} \cap \mathcal C_{k}) = 0.
\end{equation}

To estimate the contribution of large $k$'s, we use Theorem \ref{thm:DEC}
(actually at this point a weaker version of that theorem, namely Theorem 1.1 of \cite{Ch07}
is enough as the initial measure is absolutely continuous.)
Specifically, we fix the trajectory
of the first particle (that is, we fix $q_1, \varphi_1$) and for a fixed $k$ we 
choose $F = F(k)$ to be the indicator of the set $A= A(k)$,
where 
$$
A = \mathcal A_{\lambda, \delta, \varepsilon}^{\xi} ( \Phi^{k \delta}(q_1,\varphi_1)) 
= \{ (q_2,\varphi_2) : (\Phi^{k \delta}(q_1,\varphi_1),q_2,\varphi_2) \in  A_{\lambda, \delta, \varepsilon}^{\xi}\}
$$
(recall the notation (\ref{eq:Afixedq1v1v0})). That is,
$A$ is such that there is a good collision in the time interval
$[k\delta, (k+1) \delta]$ if $\Phi^{\lambda k \delta} (q_2, \varphi_2) \in A$. Clearly, $A$ is the 
$\varepsilon$ neighborhood of a two dimensional surface of area
$O(\delta)$ (see the proof of Lemma \ref{lem1}).
Thus in particular, $F$ is generalized Lipschitz (generalized
H\"older with exponent $1$) with uniformly bounded norm.
Now we apply Theorem \ref{thm:DEC} with $\alpha = 1$ to conclude
that
\begin{eqnarray}
&& (\mu \times \mu)
( \mathcal C_{0} \cap \mathcal C_{k}) \nonumber \\
&=& \int 
\mu \left( (q_2, \varphi_2) \in \mathcal A_{\lambda, \delta, \varepsilon}^{\xi}  (q_1,\varphi_1),
\Phi^{\lambda k \delta} (q_2, \varphi_2) \in \mathcal A_{\lambda, \delta, \varepsilon}^{\xi} 
( \Phi^{k \delta}(q_1,\varphi_1)) \right)  d \mu (q_1, \varphi_1) \nonumber \\
&\leq & C \int \left[
\mu \left( \mathcal A_{\lambda, \delta, \varepsilon}^{\xi}  (q_1,\varphi_1) \right)
\mu \left( \mathcal A_{\lambda, \delta, \varepsilon}^{\xi} 
( \Phi^{k \delta}(q_1,\varphi_1)) \right)
+  e^{-a \sqrt{ \lambda k \delta}} \right]
d \mu (q_1, \varphi_1). \nonumber 
\end{eqnarray}
Now using Lemma \ref{lem1b} (a)
to bound
$\mu \left( \mathcal A_{\lambda, \delta, \varepsilon}^{\xi}  (q_1,\varphi_1) \right) $
and $\mu \left( \mathcal A_{\lambda, \delta, \varepsilon}^{\xi}  ( \Phi^{k \delta}(q_1,\varphi_1)) \right)$, we conclude
$$
\sum_{ \frac{\log^{100} \varepsilon}{\delta} \leq k < \frac{\delta' }{\varepsilon \delta}}
(\mu \times \mu)
( \mathcal C_{0} \cap \mathcal C_{k}) < C
\sum_{ \frac{\log^{100} \varepsilon}{\delta} \leq k < \frac{\delta' }{\varepsilon \delta}}
[ \varepsilon^2 \delta^2 +  e^{-a\sqrt{ \lambda k \delta}}].
$$
This estimate, combined with (\ref{eq:lem:delta'lowerbd2}) and
(\ref{eq:lem:delta'lowerbd3}) yields
(\ref{eq:lem:delta'lowerbd}). We have finished the proof
of Lemma \ref{lemma:delta'}.
\end{proof}


\subsection{Proof of Theorem \ref{thm1}}

Now we want to prove a version of Lemma \ref{lemma:delta'} for a later macroscopic time interval,
conditioned on the event that there has been no good collision before. As the main difficulty is the 
lack of independence, we apply the big block small block technique to gain approximate independence 
among big blocks. The big block size will be $M   = \delta' / \varepsilon$ and the small block size is
$m = \log ^{100} \varepsilon$. 
The $n$th big block is the time interval $[m_{n-1}, M_n]$ and the $n$th small block is $[M_n,m_n]$, where
$M_n = Mn + m(n-1)$ and $m_n = (M+m)n$. 
Let us write
$$E_n = \{ \text{ there is good collision in the $n$th big block }\} \text{   and    } D_n = \cup_{N=1}^n E_N$$
Our main proposition is
\begin{proposition}
\label{lem:Markovdec}
\begin{equation} 
\label{eq:Markovdec}
\lim_{\xi \rightarrow 0}
\lim_{\delta' \rightarrow 0} \frac{1}{\delta'}
 \lim_{\varepsilon \rightarrow 0}
(\mu \times \mu)
 (E_{n+1} | \overline{D_n})
 = \rho(\lambda)
 \end{equation}
uniformly for $n < T/\delta'$.
\end{proposition} 

Note that by Lemma \ref{lem1}, 
\[
(\mu \times \mu) (\text{there is good collision in some small block}) = o_{\varepsilon} (1).
\]
and thus Theorem \ref{thm1} follows from Proposition \ref{lem:Markovdec}. It only remains to 
prove Proposition \ref{lem:Markovdec}. \vskip2mm

\noindent{\it Proof of Poposition \ref{lem:Markovdec}.}
We prove Proposition \ref{lem:Markovdec} by induction. The case $n=0$ is
Lemma \ref{lemma:delta'}. 

For general $n$ we follow a similar strategy, but
now the invariant measure is replaced by a Markov decomposition at time
$$ \tau_n= M_n + m/2.$$
Strictly speaking, we take Markov decomposition at time $\lfloor \tau_n / \mathfrak s \rfloor \mathfrak s$,
but for the ease of notation we simply write $\tau_n$ (and apply similar notation for later stopping times).
As before, we fix $q_1, \varphi_1$. 
For notational convenience given a set $F$ we denote
$$ F(q_1, \varphi_1) = \{
(q_2, \varphi_2): (q_1, \varphi_1, q_2, \varphi_2) \in F \}.$$

Now for the
fixed $q_1, \varphi_1$, let 
$\{ \ell_{n \alpha} \}_{\alpha \in \mathfrak A_n}$ be the collection of the standard pairs in the image
$\Phi^{\tau_n}_* \mu (q_2, \varphi_2)$ for which there has been no 
good collision in the first $n$ big blocks and $c_{n \alpha}$ is the relative weight of the
curve $\ell_{n \alpha}$ in this family:
\begin{equation}
\label{eq:Markovdecfirst}
\mu (A \circ \Phi^{\tau_n} | \overline{D_n} (q_1,\varphi_1)) = 
\sum_{\alpha \in \mathfrak A_n} c_{\alpha, n} \mathbb E_{\ell_{n, \alpha}} (A).
\end{equation}
Note that this Markov decomposition depends on $q_1, \varphi_1$.

First we claim that the contribution of such $q_1, \varphi_1$'s for which 
\begin{equation}
\label{eq:Markovgoodq1phi1} 
\mu \left( \overline{D_n} (q_1,\varphi_1)\right) < \delta'^2
\end{equation}
is negligible. To see this, first observe that by the inductive hypothesis,
$$ 
(\mu \times \mu) (\overline{ D_n}) \geq 1-e^{-2 T \rho (\lambda) }
$$
holds for $\xi$ small enough.
Then writing
\begin{eqnarray}
&& (\mu \times \mu) (E_{n+1} | \overline{D_n}) \nonumber \\
&=& 
\frac{1}{(\mu \times \mu) \overline{D_n}}
\int \mu \left( E_{n+1} (q_1, \varphi_1) \cap \overline{D_n}(q_1, \varphi_1) \right) d \mu(q_1, \varphi_1)
\label{eq:keylemmadecomp}
 \end{eqnarray}
we see that the contribution of $(q_1, \varphi_1)$'s satisfying
(\ref{eq:Markovgoodq1phi1}) is bounded by $C \delta'^2$ (from above, and by zero from below),
i.e. they are indeed negligible. 
Let us say that $(q_1,\varphi_1) \in G_n$ iff (\ref{eq:Markovgoodq1phi1}) is false. Now we want to apply
the growth lemma to conclude that for $(q_1,\varphi_1) \in G_n$,
\begin{equation}
\label{eq:conditionalgrowth}
\sum_{\alpha \in \mathfrak A_n \setminus \tilde{\mathfrak A}_n } c_{n,\alpha} < C \varepsilon^2/ \delta'^2
\end{equation}
where
$$
\tilde{\mathfrak A}_n = \{ \alpha \in \mathfrak A_n: |\ell_{n,\alpha} | \geq \varepsilon^2\}.
$$
Unfortunately, the growth lemma does not directly imply (\ref{eq:conditionalgrowth}),
as unstable curves may have been cut by the boundary of $\mathcal A_{\lambda, \delta, \varepsilon}^{\xi}$
in the past (depending on $q_1, \varphi_1$) and 
such fragmentations are clearly not considered in Lemma \ref{lem:gr}.
That is why we first prove


\begin{lemma}
\label{lemma:Gn'}
There is a set $G_n' \subset G_n$ such that 
\begin{enumerate}
\item $\mu (G_n \setminus G_n') < \delta'^2$
\item for any $(q_1, v_1) \in G'_n$, (\ref{eq:conditionalgrowth}) holds.
\end{enumerate}
\end{lemma} 

\begin{proof}
The only reason why (\ref{eq:conditionalgrowth}) can fail to hold is that too many 
curves have been cut by the boundary of $\mathcal A_{\lambda, \delta, \varepsilon}^{\xi}$
in the first big $n$ big blocks. Note however that by definition, 
$\tau_n - M_n = m/2$ and
 $|\Phi^{-m/2} W|$ is superpolynomially small in $\varepsilon$.
 Thus the curves that were cut before had to lie entirely in the $\varepsilon^{20}$ neighborhood
 of the boundary of $\mathcal A_{\lambda, \delta, \varepsilon}^{\xi}$ and the weight of such curves is small.
 More precisely, we define
$$ 
\mathcal B_{\lambda, t,\varepsilon} = 
\{ (q_1,\varphi_1), (q_2, \varphi_2): \exists s \in [0, t]: 
| \| \Pi_q \Phi^s (q_1, \varphi_1) - \Pi_q 
\Phi^{\lambda s} (q_2, \varphi_2) \|  - \varepsilon |< \varepsilon^{20}
\}
$$
and 
$$
\mathcal B_{\lambda, t,\varepsilon}(q_1, \varphi_1) = \{ (q_2,\varphi_2): 
(q_1,\varphi_1, q_2, \varphi_2) \in \mathcal B_{\lambda, t,\varepsilon} \}.
$$
First, we note that a simplified version of Lemma \ref{lem1} 
implies
\begin{equation}
\label{eq:conditionalgrowth1}
(\mu \times \mu) \mathcal B_{\lambda, T/\varepsilon,\varepsilon} < \varepsilon^{18}.
\end{equation}
Now let us define $G_n' \subset G_n$ by 
$$ (q_1, \varphi_1) \in G'_n \quad \text{iff } \quad 
 \mu (\mathcal B_{\lambda, T/\varepsilon,\varepsilon} (q_1, \varphi_1) )< \varepsilon^{17}.$$
 By Fubini's theorem, $\mu (G_n \setminus G_n') < \delta'^2$ holds.
 Now for a fixed $(q_1, \varphi_1) \in G'_n$ in the Markov decomposition (\ref{eq:Markovdecfirst}), let
 $\mathfrak A'_n \subset \mathfrak A_n$ be the index set of curves $W_{\alpha, n}$ for which there is some
 $k \in [m/{2 \delta}, \tau_n/\delta]$ such that
 \begin{equation}
 \label{eq:short1}
 \Phi^{-k \delta} W_{\alpha, n} \quad  \text{ intersects } \quad
 \partial \mathcal A_{\lambda, \delta, \varepsilon}^{\xi}(\Phi^{\tau_n - k \delta}(q_1,\varphi_1)).
 \end{equation}
Since 
$|\Phi^{-m/2} W_{\alpha, n}|$ is superpolynomially small in $\varepsilon$,
(\ref{eq:short1}) implies
 \begin{equation}
 \label{eq:short2}
 \Phi^{-k \delta} W_{\alpha, n} \subset 
 \mathcal B_{\lambda, \delta, \varepsilon}^{\xi}(\Phi^{\tau_n - k \delta}(q_1,\varphi_1)).
 \end{equation}
Next, $(q_1, \varphi_1) \in G'_n$ implies
\begin{equation*}
\sum_{\alpha \in \mathfrak A'_n} c_{n,\alpha} \leq \frac{1}{\mu(\overline{D_n}(q_1,\varphi_1))}
\mu (\mathcal B_{\lambda, \tau_n - m/2,\varepsilon} (q_1, \varphi_1) ) < \varepsilon^{16}.
\end{equation*}
Finally, if $\alpha \in \mathfrak A_n \setminus \mathfrak A'_n$, then $l_{n,\alpha}$ is a full curve in the image
$\Phi^{\tau_n} \mu$ and thus by the growth lemma
\begin{equation*}
\sum_{\alpha \in \mathfrak A_n \setminus \mathfrak A'_n, |\ell_{n\alpha}| < \varepsilon^2} 
c_{n,\alpha} \leq
\frac{C}{\mu(\overline{D_n}(q_1,\varphi_1))} \varepsilon^2.
\end{equation*}
Lemma \ref{lemma:Gn'} follows.
\end{proof}


By Lemma \ref{lemma:Gn'}, we have
\begin{equation}
\label{eq:keylemmaMarkov0}
\left| 
\frac{\mu \left(  
E_{n+1}(q_1, \varphi_1) \cap \overline{D_n}(q_1, \varphi_1)\right) }
{\mu \left( \overline{D_n}(q_1, \varphi_1)\right) }
 - \sum_{\alpha  \in \tilde{\mathfrak A}_n } c_{n,\alpha} 
\mathbb P_{\ell_{n,\alpha}} (E_{n+1}) 
\right| < C \varepsilon^{2}/\delta'^2.
\end{equation}
for any fixed $(q_1, \varphi_1) \in G'_n$.
We conclude that
\begin{equation}
\label{eq:Ierror}
 \left| (\mu \times \mu) (E_{n+1} | \overline{D_n}) - I
\right| < C\delta'^2 + C \varepsilon^2 \delta'^{-2},
\end{equation}
where
\begin{equation}
\label{eq:I}
I = \frac{1}{(\mu \times \mu) \overline{D_n}}
\int_{(q_1,\varphi_1)\in G'_n} \left[ \mu \left( \overline{D_n} (q_1, \varphi_1) \right) 
\sum_{\alpha  \in \tilde{\mathfrak A}_n } c_{n,\alpha} 
\mathbb P_{\ell_{n,\alpha}} (E_{n+1}) \right]d \mu(q_1, \varphi_1).
\end{equation}
Now we observe that Proposition \ref{lem:Markovdec} will be established once we prove that
\begin{equation}
\label{eq:lemmaMarkovsteps23}
\sum_{\alpha  \in \tilde{\mathfrak A}_n } c_{n,\alpha} 
\mathbb P_{\ell_{n,\alpha}} (E_{n+1}) = \delta' \rho(\lambda)(1+o_{\xi}(1)). 
\end{equation}
uniformly for $(q_1, \varphi_1) \in G_n'''$, where $G_n''' \subset G_n'$ 
is a fixed set (to be defined later) with $\mu(G'_n \setminus G_n''')< \delta'^2$. It only remains to prove 
(\ref{eq:lemmaMarkovsteps23}), which is completed in the next two lemmas.


\begin{lemma} 
({\bf Upper bound})\\
\begin{equation}
\label{eq:inductionupperbd}
  \sum_{\alpha  \in \tilde{\mathfrak A}_n } c_{n,\alpha} 
\mathbb P_{\ell_{n,\alpha}} (E_{n+1}) \leq \delta' \rho(\lambda)(1+o_{\xi}(1)).
\end{equation}
for any $(q_1, \varphi_1) \in G'_n$.
\end{lemma}

\begin{proof}
First, we introduce the notation
$$
\mathcal C'_{k} = 
\{ (q_2, \varphi_2) : 
\Phi^{\lambda k \delta} (q_2, \varphi_2) \in \mathcal A_{\lambda, \delta, \varepsilon}^{\xi} 
( \Phi^{\tau_n + k \delta}(q_1,\varphi_1))\}
$$
and write
$$
\mathbb P_{\ell_{n,\alpha}} (E_{n+1}) \leq
\sum_{k=\frac{m}{2\delta}}^{\frac{M+m/2}{\delta}} \mathbb P_{\ell_{n,\alpha}} (\mathcal C'_{k}).
$$
Now we can apply Theorem \ref{thm:DEC} (similarly to the argument in the proof of Lemma \ref{lemma:delta'})
to conclude
\begin{equation}
\label{eq:upperbdlast}
 \mathbb P_{\ell_{n,\alpha}}  (\mathcal C'_{k}) \leq \rho (\lambda) (1+ o_{\xi} (1)) 
\varepsilon \delta + C \varepsilon^{-4} e^{-a \sqrt{\lambda k \delta}},
\end{equation}
whence (\ref{eq:inductionupperbd}) follows.
\end{proof}

\begin{lemma}
({\bf Lower bound})
\label{lem:lowerbd}
\begin{equation}
\label{eq:inductionlowerbd}
  \sum_{\alpha  \in \tilde{\mathfrak A}_n} c_{n,\alpha} 
\mathbb P_{\ell_{n,\alpha}} (E_{n+1}) \geq \delta' \rho(\lambda) (1 + o_{\xi}(1)).
\end{equation}
for all $(q_1, \varphi_1) \in G_n'''$.
\end{lemma}

\begin{proof}

{\bf Step 1: Inclusion exclusion formula}

Let us introduce the notations
\[
\hat{\sum_{k}} = \sum_{k=\frac{m}{2\delta}}^{(M+m/2)/\delta} \quad \text{ and } \quad
\hat{\sum_{k_1, k_2}} = \sum_{\frac{m}{2\delta} \leq k_1 < k_2 \leq (M+m/2)/\delta} .
\]
Now we have the simple estimate
\begin{equation}
 \mathbb P_{\ell_{n,\alpha}} (E_{n+1}) 
\geq
\hat{\sum_k} \mathbb P_{\ell_{n,\alpha}}  (\mathcal C'_{k})
-
\hat{\sum_{k_1,k_2}} \mathbb P_{\ell_{n,\alpha}}  (\mathcal C'_{k_1} \cap \mathcal C'_{k_2}).
\end{equation}
Lemma \ref{lem1b} (b) and Theorem \ref{thm:DEC} imply that if
\begin{equation}
\label{eq:lowerbdgoodk}
 \dist( \Pi_q \Phi^{\tau_n + k \delta}(q_1,\varphi_1), \partial \mathcal D )> 2 \xi,
\end{equation}
then
$$
 \mathbb P_{\ell_{n,\alpha}}  (\mathcal C'_{k}) \geq \rho (\lambda) (1- o_{\xi} (1)) 
\varepsilon \delta + C \varepsilon^{-4} e^{-a \sqrt{\lambda k \delta}}.
$$
Since $\delta \ll \xi$, the set of $k$'s satisfying (\ref{eq:lowerbdgoodk}) has density $1-O(\sqrt\xi)$.
(In fact, for most orbits the density is $1-O(\xi),$ while $1-O(\sqrt\xi)$ accommodates the orbits which
are almost tangent to the boundary for many collisions).
Consequently,
$$
\hat{\sum_k} \mathbb P_{\ell_{n,\alpha}}  (\mathcal C'_{k}) \geq 
\delta' \rho(\lambda) (1-o_{\xi}(1)) + o_{\varepsilon}(1).
$$
Now (\ref{eq:inductionlowerbd}) would follow from the estimate
\begin{equation}
 \label{eq:lem:Markovdeclowerbd}
\sum_{\alpha  \in \tilde{\mathfrak A}_n } c_{n,\alpha} 
 \hat{\sum_{k_1, k_2}}
\mathbb P_{\ell_{n,\alpha}}
( \mathcal C'_{k_1} \cap \mathcal C'_{k_2}) = \delta'o_{\xi}(1).
 \end{equation}
 
Unfortunately,  (\ref{eq:lem:Markovdeclowerbd}) is not always true. However, we will prove that it is true
for a restricted set of $(q_1, \varphi_1)$'s (which we denote by $G_n'''$) and for most $\alpha$'s. 

We will need the notation
\[
\mathcal K = \left[ \frac{m}{2\delta} , (M+m/2)/\delta \right] \quad \text{ and }
\quad
\check{\cup}_{k_2} =  
\cup_{k_2=k_1+1}^{k_1+ m/\delta}
\]

{\bf Step 2: bound for $k_2-k_1<m$}

By Lemma \ref{lem3}, we have for all $k_1 \in \mathcal K$
\begin{equation*}
(\mu \times \mu) \left( \check{\cup}_{k_2}
\mathcal C_{k_1} \cap \mathcal C_{k_2} \right) < \varepsilon^{1.99}.
\end{equation*}
and consequently
\begin{equation}
 \label{eq:d'7}
\sum_{k_1 \in \mathcal K}(\mu \times \mu) \left( \check{\cup}_{k_2}
\mathcal C_{k_1} \cap \mathcal C_{k_2} \right) <\frac{\delta'}{\varepsilon} \varepsilon^{1.99}<
\varepsilon^{0.98} .
\end{equation}
Now we say that $\alpha \in \tilde{\tilde{ \mathfrak A}}_n \subset \tilde{ \mathfrak A}_n$ if 
$$|\ell_{n,\alpha} | > \varepsilon^2 \text{ and  } 
\# (\mathcal K \setminus \mathcal K' (\alpha)) < \varepsilon^{-0.5},$$ 
where
\begin{equation}
\label{eq:tildefrakA}
\mathcal K'(\alpha) = \{ k_1: 
\mathbb P_{\ell_{n, \alpha}} \left(  
\check{\cup}_{k_2}
\mathcal C'_{k_1} \cap \mathcal C'_{k_2} 
\right) < \varepsilon^{1.1} \}.
\end{equation}
Next, we define $G_n'' \subset G_n'$ as the set of such $(q_1, \varphi_1) \in G'_n$ for which 
\begin{equation}
\label{eq:badpairs}
\sum_{\alpha \in \tilde{ \mathfrak A}_n \setminus \tilde{ \tilde{ \mathfrak A}}_n} c_{n, \alpha} < \delta'^2.
\end{equation}
First we claim that $\mu(G'_n\setminus G_n'') < \delta'^2$ as needed. Assume by contradiction that 
$\mu(G'_n\setminus G_n'') > \delta'^2$. Then
\begin{eqnarray*}
 && 
 \sum_{k_1 \in \mathcal K}
 (\mu \times \mu) (\check{ \cup}_{k_2}
\mathcal C_{k_1} \cap \mathcal C_{k_2}) \\
&\geq& \int_{(q_1, \varphi_1) \in G'_n\setminus G_n''} 
\sum_{k_1 \in \mathcal K}
\mu\left(
\check{\cup}_{k_2}
\mathcal C'_{k_1} \cap \mathcal C'_{k_2}
\right)
d\mu(q_1, \varphi_1)\\
&\geq& \int_{(q_1, \varphi_1) \in G'_n\setminus G_n''} 
\mu \left( \overline{D_n} (q_1,\varphi_1)\right)
\sum_{\alpha \in \tilde{ \mathfrak A}_n} c_{n, \alpha} \sum_{k_1 \in \mathcal K}
\mathbb P_{\ell_{n, \alpha}}  \left( 
\check{\cup}_{k_2}
\mathcal C'_{k_1} \cap \mathcal C'_{k_2}
\right) d\mu(q_1, \varphi_1)\\
&\geq& \int_{(q_1, \varphi_1) \in G'_n\setminus G_n''} 
\mu \left( \overline{D_n} (q_1,\varphi_1)\right)
\sum_{\alpha \in \tilde{ \mathfrak A}_n \setminus \tilde{\tilde{ \mathfrak A}}_n} c_{n, \alpha} 
\sum_{k_1 \in \mathcal K \setminus \mathcal K'(\alpha)}
\mathbb P_{\ell_{n, \alpha}}  \left( 
\check{\cup}_{k_2}
\mathcal C'_{k_1} \cap \mathcal C'_{k_2}
\right) d\mu(q_1, \varphi_1).
\end{eqnarray*}
By the definition of $G'_n$, $G_n''$, $ \tilde{\tilde{ \mathfrak A}}_n$ and $\mathcal K'(\alpha)$,
we see that this last expression is bigger than $\delta'^4 \varepsilon^{0.6}$
which is a contradiction with (\ref{eq:d'7}). Thus $\mu(G'_n\setminus G_n'') < \delta'^2$ indeed holds.

For $\alpha  \in \tilde{\tilde{\mathfrak A}}_n$, we use the estimate
\begin{eqnarray}
 \hat{\sum_{k_1, k_2}}
\mathbb P_{\ell_{n,\alpha}} 
( \mathcal C'_{k_1} \cap \mathcal C'_{k_2}) 
&<&
\hat{\sum_{k_1}}
\sum_{ k_2 =k_1 + m /\delta}^{ (M+m)/\delta}
\mathbb P_{\ell_{n,\alpha}} ( \mathcal C'_{k_1} \cap \mathcal C'_{k_2}) \label{eq:offdiag}\\
&+&
\log^{100} \varepsilon
\sum_{k_1 \in \mathcal K'(\alpha)} 
\mathbb P_{\ell_{n,\alpha}} ( \check{\cup}_{k_2} \mathcal C'_{k_1} \cap \mathcal C'_{k_2}) \label{eq:diagK'}\\
&+& \log^{100} \varepsilon
\sum_{k_1 \in \mathcal K \setminus \mathcal K'(\alpha)}
\mathbb P_{\ell_{n,\alpha}} ( \mathcal C'_{k_1} ). \label{eq:diagK-K'}
 \end{eqnarray}
By the definition of $\mathcal K'(\alpha)$, (\ref{eq:diagK'}) is bounded by $ \varepsilon^{0.05}$.
Since  $\alpha  \in \tilde{\tilde{\mathfrak A}}_n$
and by Theorem \ref{thm:DEC},  (\ref{eq:diagK-K'}) is bounded by $\varepsilon^{0.4}$.\\

{\bf Step 3: bound for $k_2-k_1\geq m$}

In order to estimate (\ref{eq:offdiag}), we use Markov decomposition at time 
$\tau_{n, k_1} := \tau_n+k_1 \delta +\frac{m}{2}$ conditioned on $\mathcal C'_{k_1}$
\begin{equation}
\label{eq:Markovdecsecond}
\mathbb P_{\ell_{n,\alpha}} (A \circ \Phi^{k_1 + m /2} | \mathcal C'_{k_1})=
\sum_{\beta \in \mathfrak B_{n,\alpha, k_1}} c_{ n, \alpha, k_1, \beta} 
\mathbb E_{\ell_{n, \alpha, k_1, \beta}} (A).
\end{equation}
By (\ref{eq:upperbdlast}), we have $\mathbb P_{\ell_{n \alpha}} (\mathcal C'_{k_1}) <  C \delta \varepsilon $.
Now we want to guarantee that the short curves in $\mathfrak B_{n, \alpha, k_1}$ have small weight,
at least for most $\alpha$'s. 

Using  (\ref{eq:conditionalgrowth1}), we see that
$$
\int_{(q_1,\varphi_1 ) \in G_n''} 
\sum_{\alpha \in \tilde{\tilde{\mathfrak A}}_n} c_{n,\alpha}  \mathbb P_{\ell_{n \alpha}}( \mathcal D'_{k_1})
d \mu(q_1,v_1) < \varepsilon^{17},
$$
where $\mathcal D'$ is defined as $\mathcal C'$ with $\mathcal A$ replaced by $\mathcal B$.
Now we define $G_{n,k_1}''' \subset G_n''$ as the set of $(q_1, \varphi_1) $'s for which
$$
\sum_{\alpha \in \tilde{\tilde{\mathfrak A}}_n} c_{n,\alpha}  \mathbb P_{\ell_{n \alpha}}( \mathcal D'_{k_1})
< \varepsilon^{14}.
$$
and $G_n''' = \cap_{k_1} G_{n,k_1}'''$.
By Fubini's theorem, $\mu (G_n'' \setminus G_n''')  < \varepsilon$. From now on, we assume
$(q_1, \varphi_1) \in G_n'''$.

Next, we define $\tilde{\tilde{\tilde{\mathfrak A}}}_n \subset \tilde{\tilde{\mathfrak A}}_n$ as the 
set of such $\alpha$'s for which
\[
\hat{\sum_{k_1}} \mathbb P_{\ell_{n \alpha}}( \mathcal D'_{k_1})
< \varepsilon^{11}.
\]
Again by Fubini's theorem,
\begin{equation}
\label{eq:tilde3}
\sum_{\alpha \in \tilde{\tilde{\mathfrak A}}_n \setminus \tilde{\tilde{\tilde{\mathfrak A}}}_n}
c_{n,\alpha} < \varepsilon.
\end{equation}
Now we can repeat the second half of the proof of Lemma \ref{lemma:Gn'} to conclude
that for $(q_1,\varphi_1) \in G'''_n$ and for $\alpha \in \tilde{\tilde{\tilde{\mathfrak A}}}_n$,
$$
\sum_{\beta \in \mathfrak B_{n,\alpha, k_1}, |\ell_{n, \alpha, k_1, \beta}| < \varepsilon^{-6}} 
c_{ n, \alpha, k_1, \beta} 
 < \varepsilon^{-4}.
$$
Now if $|\ell_{n, \alpha, k_1, \beta}| > \varepsilon^{-6}$, 
we use the same argument as in (\ref{eq:upperbdlast}) to conclude
$$
\mathbb P_{\ell_{n, \alpha, k_1, \beta}} 
( \mathcal C''_{k_2-k_1 - m/(2 \delta)}) < C \delta \varepsilon,
$$
where
$$
\mathcal C''_{k} = 
\{ (q_2, \varphi_2) : 
\Phi^{\lambda k \delta} (q_2, \varphi_2) \in \mathcal A_{\lambda, \delta, \varepsilon}^{\xi} 
( \Phi^{\tau_{n,k_1} + k \delta}(q_1,\varphi_1))\}.
$$
Hence for $\alpha \in \tilde{\tilde{\tilde{\mathfrak A}}}_n$,
(\ref{eq:offdiag}) is bounded by
$$
\hat{\sum_{k_1}} \mathbb P_{\ell_{n \alpha}} (\mathcal C'_{k_1})
\sum_{ k_2 =k_1 + m /\delta}^{ (M+m)/\delta}
\sum_{\beta \in \mathfrak B_{n,\alpha, k_1}} 
c_{n,\alpha, k_1, \beta}
\mathbb P_{\ell_{n, \alpha, k_1, \beta}} 
( \mathcal C''_{k_2-k_1 - m/(2 \delta)}) < C \delta'^2.
$$

We conclude
\begin{equation}
 \label{eq:lem:Markovdeclowerbd2}
\sum_{\alpha  \in \tilde{\tilde{\tilde{\mathfrak A}}}_n } c_{n,\alpha} 
 \hat{\sum_{k_1, k_2}}
\mathbb P_{\ell_{n,\alpha}}
( \mathcal C'_{k_1} \cap \mathcal C'_{k_2}) = \delta'o_{\xi}(1).
 \end{equation}

{\bf Step 4: Finishing the proof }\\
By (\ref{eq:badpairs}) and (\ref{eq:tilde3}), we can 
replace 
(\ref{eq:lem:Markovdeclowerbd})
in Step 1 by (\ref{eq:lem:Markovdeclowerbd2}). Lemma \ref{lem:lowerbd} follows.
We have finished the proof of Proposition \ref{lem:Markovdec}
and Theorem \ref{thm1}.
\end{proof}


\end{document}